\documentclass{article}
\usepackage{graphicx} 
\usepackage{multirow}
\usepackage{amsmath}
\usepackage{relsize}
\usepackage{amssymb}
\usepackage{amsthm}
\usepackage{enumerate}
\usepackage{subfigure}
\usepackage{subcaption}
\usepackage{natbib}

\usepackage{tikz}
\usepackage{float}
\usetikzlibrary{arrows, positioning}
\usepackage{nicematrix}
\usepackage{graphicx}

\usepackage{amsmath, amssymb}
\usepackage{etoolbox}

\newtheorem{theorem}{Theorem}
\newtheorem{lemma}[theorem]{Lemma}
\newtheorem{proposition}[theorem]{Proposition}
\newtheorem{corollary}{Corollary}
\newtheorem{example}{Example}

\theoremstyle{definition}
\newtheorem{definition}{Definition}

\theoremstyle{remark}

\newtheorem*{proofofgcd}{Proof of Theorem \ref{thm:gcd}}
\usepackage[utf8]{inputenc}

\usepackage{bm}
\usepackage{tikz}
\usetikzlibrary{arrows.meta, positioning}
\usetikzlibrary{graphs, graphs.standard}

\usepackage{eucal}
\usepackage{blkarray}
\usepackage{natbib}

\begin{document}

\title{Grothendieck group of the Leavitt path algebra over power graphs of  prime-power cyclic groups}
\author{Aslı Güçlükan İlhan, Müge Kanuni and Ekrem Şimşek}
\date{November 2024}

\maketitle

\begin{abstract}
In this paper, the Grothendieck group of the Leavitt path algebra over the power graphs of all prime-power cyclic groups is studied by using a well-known computation from linear algebra. More precisely, the Smith normal form of the matrix derived from the adjacency matrix associated with the power graph of prime-power cyclic group is calculated.
\end{abstract}

\section{Introduction}

Graphs are essential tools in solving many mathematical problems. In this paper, we use the joyous interplay of algebra and graph theory back and forth. Starting with a cyclic  group of prime-power $G$, its directed power graph (power digraph) $Pow(G)$ is constructed. The power digraph of a group was introduced by Kelarev and Quinn as a combinatorial tool in \cite{kelarev2000combinatorial}.
Many interesting and fruitful results appeared in literature both on groups and semigroups using the power digraph notion on the algebraic structure. The power digraph obtained by removing the vertex representing the identity of the group, so called the punctured power digraph of a group, was more useful as a tool to study, since the identity is always a sink of the power digraph for any group given. For a survey on power graphs, we refer the reader to \cite{abawajy2013power,Cameron}.
 On the other hand, Leavitt path algebras  originated from its connection with non-IBN algebras of Leavitt \cite{leavitt1962module} and graph C*-algebras \cite{ara2007nonstable}. Leavitt path algebra proved to be a productive avenue of research in the last two decades. A concise book on the Leavitt path algebra and its connections is \cite{abrams2017leavitt}.

Recently, a study of Leavitt path algebras constructed over power digraphs of certain finite groups were conducted by Das, Sen and Maity in \cite{das2022leavitt}. These results reveal the connection between algebraic properties of the Leavitt path algebra and the graph theoretical properties of the (punctured) power digraph of certain finite groups. In particular, it is stated that the Leavitt path algebra over the punctured power digraph of a finite group is purely infinite simple if and only if it is isomorphic to a cyclic group $\mathbb{Z}_p$  for some prime $p > 3$ . 
Consequently, the same authors studied the Grothendieck group of the Leavitt path algebra associated with the punctured power digraph of  the cyclic group $\mathbb{Z}_p$ $(p>3)$ in \cite{das2023grothendieck}.

As a direct consequence of Das et al., one can easily determine the Grothendieck group associated to the direct sum of $n$ copies of a prime power group $\mathbb{Z}_{p}$ with $p\geq5$ (See Proposition \ref{prop.3}).  In this paper, we extended these results to determine the Grothendieck group of the Leavitt path algebra associated with the punctured power digraph of the cyclic group $\mathbb{Z}_{p^n}$ for any prime $p$ and integer $n$ (Section \ref{sect:SNF_for_pgeq3}, Theorem \ref{thm:Main_pgeq3} and Section \ref{sec:pow2n}, Theorem \ref{thm.pow2n}).

The outlay of the paper is as follows: 

In Section \ref{sec:Preliminaries}, we give the necessary definitions and preliminary results on Leavitt path algebras and their Grothendieck groups. Our approach in finding the Grothendieck group of the Leavitt path algebra is from linear algebra and the well-known calculation  stated in Theorem \ref{thm.2}. We compute the cokernel of the transpose of the corresponding matrix derived from the power digraph by using the Smith Normal Form of this matrix. We also present two observations on the Leavitt path algebras and the Grothendieck groups of the Leavitt path algebras of punctured power digraphs, corresponding to groups with exponent $3$ (Proposition \ref{prop.1}) and to arbitrary elementary abelian groups (Proposition \ref{prop.3}), respectively. 

In Section \ref{sect:SNF_for_pgeq3}, we state one of the main theorems of this paper (Theorem \ref{thm:Main_pgeq3}) that gives the Grothendieck group of Leavitt path algebra associated with the punctured power digraph of cyclic group of order $p^n$ for an odd prime $p$, up to isomorphism. To obtain this result, we compute the greatest common divisor of $k\times k$ minors of the matrix $M(p^n)$ (see Definition \ref{def:thematrix}) for $1 \leq k \leq p^{n}-1$. Since they are very detailed, Section \ref{sect:proof_Thm} is devoted to them. In Section \ref{sec:pow2n}, we use the same idea to derive the Grothendieck group of Leavitt path algebra of the directed punctured power graph of the cyclic group of order $2^{k}$ and obtain Theorem \ref{thm.pow2n}.

\section{Preliminaries on Leavitt Path Algebras}
\label{sec:Preliminaries}

A \textit{directed graph (digraph)} $E=(E^{0},E^{1},r,s)$ consists of a set $E^{0}$ of vertices, a set $E^{1}$ of edges and maps $r,s:E^{1} \to E^{0}$ called range and source maps, respectively. An edge $e$ with $s(e)=u$ and $r(e)=v$ is represented as an arrow from the vertex $u$ to the vertex $v$. A vertex $v \in E^0$ is called a \textit{sink} if $s^{-1}(v)= \emptyset$. A vertex $v$ with $0<|s^{-1}(v)|< \infty$ is called \textit{regular}. The \textit{adjacency matrix} of $E$ is the matrix $A_E \in \mathbb{Z}^{E^0 \times E^0}$ whose $(u,v)$-entry is  $|\{e \in E^{1}| s(e)=u, r(e)=v \}|$. 
The \textit{ power digraph} $Pow(G)$ of a group $G$ is the graph whose vertices are the elements of $G$, with an edge from $x$ to $y$ if $x \neq y$ and $y=x^m$ for some integer. In $Pow(G)$, the vertex corresponding to the identity element is a sink. The graph obtained by removing this vertex is called a \textit{punctured power digraph} of $G$ and denoted by $Pow^{\ast}(G)$. Throughout the paper, we refer to the digraphs $Pow(G)$ and $Pow^{\ast}(G)$ as the power graph of $G$ and the punctured power graph of $G$, respectively. See Figure \ref{figure1} and Figure \ref{figure2} for examples of such graphs.
\vspace{1em}

    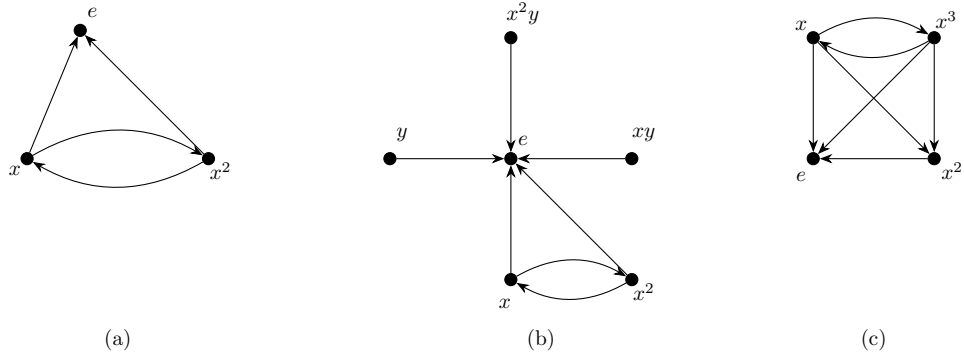
\begin{figure}[H]
    
        \centering

        \begin{tikzpicture}[>=Stealth, node distance={3cm}, 
                    main/.style = {circle, draw, fill=black, minimum size=2mm, inner sep=0pt}, 
                    every edge/.style={draw,thick,->,>=Stealth},
                    scale=.8, transform shape]

    \node[main] (0) at (0,0) {};
    \node[main] (1) [right of=0] {};
    \node[main] (2) [above left of=1] {};

    \node at (2) [xshift=0.2cm, yshift=0.3cm] {$e$};
    \node at (0)[xshift=-0.2cm,yshift=-0.2cm] {$x$};
    \node at (1) [xshift=0.2cm,yshift=-0.2cm] {$x^2$};

    \draw[->] (0) to (2);
    \draw[->] (1) to (2);
    \draw[->] (0) to [bend left] (1);
    \draw[->] (1) to [bend left] (0);

    \node (z) at (1.5,-3) {(a)}; 

    \node[main](3) at (8,0) {};
    \node[main](4) at (10,0) {};  
    \node[main](5) at (6,0) {};  
    \node[main](6) at (8,-2) {};  
    \node[main](7) at (8,2) {};
    \node[main](8) at (10,-2) {};

    \node at (3) [xshift=0.2cm, yshift=0.3cm] {$e$};
    \node at (4) [xshift=0.2cm, yshift=0.4cm] {$xy$};
    \node at (5) [xshift=0.2cm, yshift=0.4cm] {$y$};
    \node at (6) [xshift=-0.1cm, yshift=-0.4cm] {$x$};
    \node at (7) [xshift=0.2cm, yshift=0.4cm] {$x^2y$};
    \node at (8) [xshift=0.2cm, yshift=-0.2cm] {$x^2$};

     \node (u) at (8.5,-3) {(b)}; 

    \draw[->]  (4) to (3);
    \draw[->]  (5) to (3);
    \draw[->]  (6) to (3);
    \draw[->]  (7) to (3);
    \draw[->]  (8) to (3);
    \draw[->]  (6) to [bend left] (8);
    \draw[->]  (8) to [bend left] (6);

    \node[main] (9) at (13,0) {};
    \node[main] (10) at (13,2) {};
    \node[main] (11) at (15,0) {};
    \node[main] (12) at (15,2) {};

    \node at (9) [xshift=-0.2cm, yshift=-0.3cm] {$e$};
    \node at (10) [xshift=-0.2cm, yshift=0.2cm] {$x$};
    \node at (11) [xshift=0.3cm, yshift=-0.2cm] {$x^2$};
    \node at (12) [xshift=0.2cm, yshift=0.3cm] {$x^3$};

    \node (v) at (14,-3) {(c)}; 

    \draw[->] (10) to (9);
    \draw[->] (11) to (9);
    \draw[->] (12) to (9);
    \draw[->]  (10) to [bend left] (12);
    \draw[->]  (12) to [bend left] (10);
    \draw[->]  (10) to (11);
    \draw[->] (12) to (11);
    
\end{tikzpicture}
        
        \caption{Power graphs of the groups (a) $\mathbb{Z}_{3}$, (b) $D_{3}$  and (c) $\mathbb{Z}_{4}$, respectively.}
        \label{figure1}
    \end{figure}

\vspace{1em}

     \begin{figure}[h]
     
        \centering

        \begin{tikzpicture}[>=Stealth, node distance={3cm}, 
                    main/.style = {circle, draw, fill=black, minimum size=2mm, inner sep=0pt}, 
                    every edge/.style={draw,thick,->,>=Stealth},
                    scale=.8, transform shape]

    \node[main] (0) at (0,0) {};
    \node[main] (1) [right of=0] {};

    \node at (0)[xshift=-0.2cm,yshift=-0.2cm] {$x$};
    \node at (1) [xshift=0.2cm,yshift=-0.2cm] {$x^2$};
    \node (z) at (1.5,-1.5) {(a)};

    \draw[->] (0) to [bend left] (1);
    \draw[->] (1) to [bend left] (0);

    \node[main](4) at (10,1) {};  
    \node[main](5) at (6,1) {};  
    \node[main](6) at (7,0) {};  
    \node[main](7) at (8,1) {};
    \node[main](8) at (9,0) {};

    \node (u) at (8,-1.5) {(b)};
    \node at (4) [xshift=0.2cm, yshift=0.4cm] {$xy$};
    \node at (5) [xshift=0.2cm, yshift=0.4cm] {$y$};
    \node at (6) [xshift=-0.1cm, yshift=-0.4cm] {$x$};
    \node at (7) [xshift=0.2cm, yshift=0.4cm] {$x^2y$};
    \node at (8) [xshift=0.2cm, yshift=-0.2cm] {$x^2$};

    \draw[->]  (6) to [bend left] (8);
    \draw[->]  (8) to [bend left] (6);

    \node[main] (10) at (13,1) {};
    \node[main] (11) at (15,-1) {};
    \node[main] (12) at (15,1) {};

    \node (v) at (14,-1.5) {(c)}; 
    \node at (10) [xshift=-0.2cm, yshift=0.2cm] {$x$};
    \node at (11) [xshift=0.3cm, yshift=-0.2cm] {$x^2$};
    \node at (12) [xshift=0.2cm, yshift=0.3cm] {$x^3$};

    \draw[->]  (10) to [bend left] (12);
    \draw[->]  (12) to [bend left] (10);
    \draw[->]  (10) to (11);
    \draw[->] (12) to (11);
    
\end{tikzpicture}
        
        \caption{ Punctured power graphs of the groups (a) $\mathbb{Z}_{3}$, (b) $D_{3}$  and (c) $\mathbb{Z}_{4}$, respectively.}
        \label{figure2}

    \end{figure}
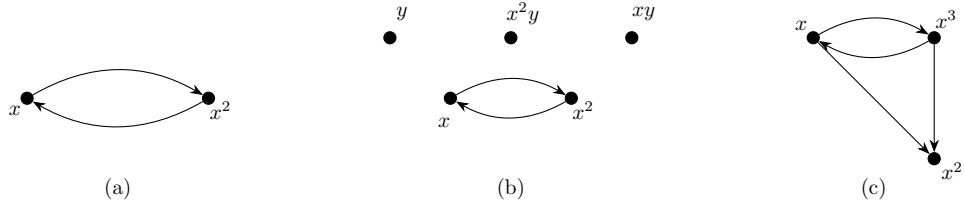

\vspace{1em}

In this paper, we study the Groethendieck groups of Leavitt path algebras of graphs $Pow(G)$ and $Pow^{\ast}(G)$ for cyclic groups of prime-power order. Our main reference for Leavitt path algebras is \cite{abrams2017leavitt}.

\begin{definition}
Given a directed graph $E$ and a field $K$, the \textit{Leavitt path algebra} of $E$ with coefficients in $K$, denoted by $L_{K}(E)$, is the free associative $K$-algebra generated by $ E^{0}\cup E^{1}\cup (E^{1})^{\ast} $ satisfying the following relations
\begin{itemize}
    \item[(V)] $vv'=\delta_{v,v'}v$ for all $ v,v' \in E^{0}$,
    \item[(E1)] $s(e)e=er(e)=e$ for all $e \in E^{1}$,
    \item[(E2)] $r(e^{\ast})e^{\ast}=e^{\ast}s(e^{\ast})=e^{\ast}$ for all $ e^{\ast} \in (E^{1})^{\ast}$,
    \item[(CK1)] $e^{\ast}e'=\delta_{e,e'}r(e)$ for all $ e,e' \in E^{1}$,
    \item[(CK2)] $v=\underset{\{ e\in E^{1}| s(e)=v\}}{\sum}ee^{\ast}$ for every regular vertex $v\in E^{0}$. 
\end{itemize}         
\end{definition}

In \cite{das2022leavitt}, the authors show that the Leavitt path algebra of $Pow(G)$ is simple if and only if $G$ is an elementary abelian $2$-group of order $n$, i.e, $G\cong \mathbb{Z}_2^r$ with $n=2^r$. Here, $L_{K}(Pow(\mathbb{Z}_2^r))$ is isomorphic to the matrix algebra $M_{2^r}(K)$ since the graph $Pow(\mathbb{Z}_2^r)$ is of the following form
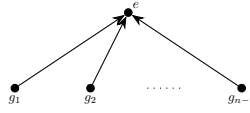
\begin{center}
     \begin{tikzpicture}[>=Stealth, node distance={3cm}, 
                            main/.style = {circle, draw, fill=black, minimum size=2mm, inner sep=0pt}, 
                            every edge/.style={draw,thick,->,>=Stealth},
                            scale=.5, transform shape]

            \node[main] (1) at (0,0) {};
            \node[main] (2) at (2,0) {};
            \node[main] (3) at (6,0) {};
            \node[main] (4) at (3,2) {};
            
            \draw[dotted] (3.5,0) -- (4.5,0);
            \node at (1)[yshift=-0.3cm] {$g_1$};
            \node at (2) [yshift=-0.3cm] {$g_2$};
            \node at (3)[yshift=-0.3cm] {$g_{n-1}$};
            \node at (4) [xshift=0.2cm,yshift=0.2cm] {$e$};
            \draw[->] (1) to  (4);
            \draw[->] (2) to (4);
            \draw[->] (3) to  (4);
\end{tikzpicture} 
\end{center} where $e$ is the identity element of the group and $g_1, \dots, g_{n-1}$ are the non-identity elements. In this case, the punctured power graph consists of $n-1$ isolated vertices and hence $L_{K}(Pow^{\ast}(\mathbb{Z}_2^r)) \cong K^{n-1}$. When $G$ is an elementary abelian $3$-group of order $m=3^r$, the punctured power graph of $G$ is a disjoint union of $\frac{m-1}{2}$ copies of the graph  \begin{tikzpicture}[>=Stealth, node distance={3cm}, 
                            main/.style = {circle, draw, fill=black, minimum size=2mm, inner sep=0pt}, 
                            every edge/.style={draw,thick,->,>=Stealth},
                            scale=.5, transform shape]
            \node[main] (0) at (0,0) {};
            \node[main] (1) at (2,0) {};
            \draw[->] (0) to [bend left] (1);
            \draw[->] (1) to [bend left] (0);
\end{tikzpicture}. Here, the Leavitt path algebra of its connected components are all isomorphic to $M_2(K[x,x^{-1}])$ \cite{Bjerregaardvd} and hence $L_{K}(Pow^{\ast}(\mathbb{Z}_3^r))$ is isomorphic to the direct sum of $\frac{m-1}{2}$ copies of $M_2(K[x,x^{-1}])$. Recall that the exponent of a group is the least common multiple of the orders of its elements. Therefore, when $G$ is a group of exponent 3, its punctured power graph is isomorphic to that of elementary abelian $3$-group of the same order. Hence we can generalize this observation as follows.
\begin{proposition}
\label{prop.1}
The Leavitt path algebra of a punctured power graph of a group of order $m$ and exponent $3$ is isomorphic to $\underset{\frac{m-1}{2}}{\bigoplus}M_2(K[x,x^{-1}]).$
    
\end{proposition}

The power graph of the cyclic group of order $3$ and the punctured power graph of the cyclic group of order $4$ are isomorphic to the following graph
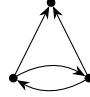
\begin{center}
     \begin{tikzpicture}[>=Stealth, node distance={3cm}, 
                            main/.style = {circle, draw, fill=black, minimum size=2mm, inner sep=0pt}, 
                            every edge/.style={draw,thick,->,>=Stealth},
                            scale=.5, transform shape]

            \node[main] (1) at (0,0) {};
            \node[main] (2) at (2,0) {};
            \node[main] (3) at (1,2) {};

            \draw[->] (1) to  (3);
            \draw[->] (2) to (3);
            \draw[->] (1) [bend left] to  (2);
            \draw[->] (2) [bend left] to  (1);
\end{tikzpicture} 
\end{center} whose Leavitt path algebra is unknown. One can easily show that  the Grothendieck group of the Leavitt path algebra of this graph is isomorphic to $\mathbb{Z}\bigoplus\mathbb{Z}_2$ by using the following theorem.

Let $K_0(L_K(E))$ denote the Grothendieck group of the Leavitt path algebra of the graph $G$. Let $A_{ns}$ and $I_{ns}$ be the matrices obtained by removing the rows corresponding to sinks of $E$ from the adjacency matrix $A_E$ of $E$ and the $E^0\times E^0$ identity matrix, respectively.

\begin{theorem}[Theorem 6.1.9, \cite{abrams2017leavitt}]
\label{thm.2}
Let \( E \) be a row-finite graph. Then
\[
K_0(L_K(E)) \cong \mathrm{Coker}\Big((I_{ns} - A_{ns})^{tr} : \bigoplus_{v \in \text{Reg}(E)} \mathbb{Z} \to \bigoplus_{v \in E^0} \mathbb{Z}\Big).
\]
\end{theorem}

Recall that given an $m\times n$ matrix $M$ over $\mathbb{Z}$ of rank $r$,  the Smith normal form of $M$ is an $m\times n$ diagonal matrix $S=[s_{ij}]$ equivalent to $M$. Here, $s_{ij}=0$ for $i\neq j$ and $i=j > r$ and $s_{11}s_{22}\cdots s_{kk}$ is the greatest common divisor of the $k\times k$ minors of $M$ for $1\leq k \leq r$. It is well-known that 
$$\mathrm{coker}(M) \cong \mathbb{Z}^{n-r}\oplus\mathbb{Z}/ s_{11}\mathbb{Z} \oplus \cdots \oplus \mathbb{Z}/ s_{rr}\mathbb{Z}.$$
For a graph with $c$ sinks, the number of rows of $(I_{ns} - A_{ns})^{tr}$ is $c$ more than the number of columns. Therefore we immediately have the following result.
\begin{corollary}
\label{cor:mainref} Let \( E \) be a finite graph with $n$ vertices and $c$ sinks. Let   $S$ be the Smith normal form of the matrix $(I_{ns} - A_{ns})^{tr}$ of rank $r$ with non-zero diagonal entries $ s_1, s_2, \ldots, s_{r}$. Then
$$ K_0(L_K(E)) \cong \mathbb{Z}^{n-r}\oplus \mathbb{Z}/s_1\mathbb{Z} \oplus \mathbb{Z}/s_2\mathbb{Z} \oplus \cdots \oplus \mathbb{Z}/s_{r}\mathbb{Z}. $$ 
\end{corollary}
By using this corollary, Das et al. \cite{das2023grothendieck} show that for any prime $p \geq 5$, the Grothendieck group of $L_{K}(Pow^{\ast}(\mathbb{Z}_{p}))$ is isomorphic to $\underbrace{\mathbb{Z}_{2} \oplus \cdots \oplus \mathbb{Z}_{2}}_{(p-3) \ \mathrm{times}} \oplus   \mathbb{Z}_{2p-6}$. Let $p\geq 5$ be a prime. Since the punctured power graph of $\mathbb{Z}^n_{p}$ is a disjoint union of $\frac{p^{n}-1}{p-1}$ copies of the punctured power graph of $\mathbb{Z}_{p}$, one can directly obtain the following result. 
\begin{proposition}
\label{prop.3}
     For any prime $p \geq 5$, the Grothendieck group of $L_{K}(Pow^{\ast}(\mathbb{Z}^n_{p}))$ is isomorphic to $\underbrace{\mathbb{Z}_{2} \oplus \cdots \oplus \mathbb{Z}_{2}}_{m(p-3) \ \mathrm{times}} \oplus \underbrace{ \mathbb{Z}_{2p-6} \oplus \cdots \mathbb{Z}_{2p-6}}_{m \ \mathrm {times}}$ where $m=\frac{p^{n}-1}{p-1}$. 

 \end{proposition}

\section{Grothendieck group of $L_K(Pow^{\ast}(\mathbb{Z}_{p^n }))$ for $p \geq 3$}
\label{sect:SNF_for_pgeq3}
Let $p$ be an odd prime number and $\mathbb{Z}_{p^n}$ be a cyclic group of order $p^n$. For a positive integer $1\leq a \leq p^n-1$, denote the vertex corresponding to $a$ in $Pow^{\ast}(\mathbb{Z}_{p^{n}})$ by $v_a$. Then the vertex set of $Pow^{\ast}(\mathbb{Z}_{p^{n}})$ is $\{v_1,\dots,v_{p^n-1}\}$ and there is an edge from $v_a$ to $v_b$ if and only if $a\neq b$ and the order of $b$ divides the order of $a$. Since the order of $a$ is $ \dfrac{p^n}{\mathrm{gcd} \{a,p^n\}}$, there are $p^i-p^{i-1}$ elements of order $p^i$. Let us reorder the vertices of $Pow^{\ast}(\mathbb{Z}_{p^{n}})$ as $w_1,\cdots, w_{p^n -1}$ where $\{w_{p^{i-1}},\dots, w_{p^{i}-1}\}$ is the set of vertices of order $p^i$. In this order, the adjacency matrix of $Pow^{\ast}(\mathbb{Z}_{p^{n}})$ is of the following form

\setlength{\arraycolsep}{.2pt} 
\renewcommand{\arraystretch}{1.3} 
\begin{equation}
\label{eq:adjacency_matrix}
A_{Pow^{\ast}(\mathbb{Z}_{p^{n}})}=    \NiceMatrixOptions{xdots/shorten=0.5em}
 \begin{pNiceMatrix} 
    S_{p-1} & 0 &0 & \ \cdots &  0  \\
     1 &S_{p^2-p} & 0 &\ \cdots & 0\\
    1 & 1 &S_{p^3-p^2}  &\ \cdots & 0\\
    \vdots & \vdots & \vdots & \ddots & 0\\
    1& 1& 1& \ \cdots & \ \ S_{p^n-p^{n-1}}
    \end{pNiceMatrix} \end{equation}
where $S_{k}$ is the circulant matrix associated to the vector $(0,1,\dots,1) \in \mathbb{Z}^{k}$. Recall that the circulant matrix $C$ associated to the vector $v=(v_1,\dots,v_n)$ is the $n\times n$ square matrix of the form
\[ C=\begin{pNiceMatrix}[columns-width=0.8cm] v_1& v_2 & v_3 & \cdots& v_n \\
v_{n}& v_1& v_2 & \cdots& v_{n-1} \\
v_{n-1}& v_{n}& v_1 & \cdots& v_{n-2}\\
\vdots & \vdots & \vdots&\ddots &\vdots\\
v_2& v_3 & v_4 & \cdots& v_1
\end{pNiceMatrix}.\]

In the same order, $I-A^{tr}_{Pow^{\ast}(\mathbb{Z}_{p^{n}})}$ is
the matrix
\setlength{\arraycolsep}{.2pt} 
\renewcommand{\arraystretch}{1.3} 
\begin{equation}\label{eq:thematrix}
   \begin{pNiceMatrix} 
    B_{p-1} & -1  &-1 & \ \cdots &  -1  \\
     0 &B_{p^2-p} & -1 &\ \cdots & -1\\
     0 & 0 &B_{p^3-p^2}  &\ \cdots & -1\\
    \vdots & \vdots & \vdots & \ddots & -1\\
    0& 0& 0& \ \cdots & \ \ B_{p^n-p^{n-1}}
    \end{pNiceMatrix} \end{equation}
where $B_{k}$ is the circulant matrix associated to the vector $(1,-1,\ldots,-1) \in \mathbb{Z}^k$. By Corollary \ref{cor:mainref}, to obtain the Grothendieck group of the Leavitt path algebra over the graph $Pow^{\ast}(\mathbb{Z}_{p^{n}})$, it suffices to find the Smith normal form of the above matrix. For this, one needs to calculate the greatest common divisors of the determinant of $(k \times k)$ minors of it.

\begin{theorem}\label{thm:gcd}
    The greatest common divisor of the determinants of $(k \times k)$ minors of $I-A^{tr}_{Pow^{\ast}(\mathbb{Z}_{p^{n}})}$ is equal to
    \[=\begin{cases}
        1, & \text{for} \ 1\leq k \leq n, \\
        2^{k-n}, & \text{for} \ n+1 \leq k \leq p^n-n-1,\\
        2^{2k-p^{n}+1}, &\text{for}\ p^n-n \leq k \leq p^n-2, \\
       |\prod^{n}_{i=1} (2-\varphi(p^{i})) 2^{\varphi(p^{i})-1}|, & \text{for}\ k=p^n-1. 
       
    \end{cases}
   \] where $\varphi$ is the Euler's phi function.
\end{theorem}

The proof of the above theorem is given in Section \ref{sect:proof_Thm}.
As the non-zero diagonal entries of the Smith Normal form of $I-A^{tr}_{Pow^{\ast}(\mathbb{Z}_{p^{n}})}$ are obtained by dividing each number given in the above theorem by previous one, we have the following results.

 \begin{corollary}
 \label{cor:snfforprime}
If $p\geq 3$ is a prime, then the Smith normal form of the matrix $I-A^{tr}_{Pow^{\ast}}(\mathbb{Z}_{p^{n}})$ is equal to $\mathrm{diag}(\underbrace{1,1,\dots,1}_{n \ \mathrm{times}},\underbrace{2,2,\dots,2}_{p^{n}-2n-1 \ \mathrm{times}},\underbrace{4,4,\dots,4}_{n-1 \ \mathrm{times}},|\eta|)$,
where $\eta= \dfrac{\overset{n}{ \underset{i=1}{\prod}} (2-\varphi(p^{i})) 2^{\varphi(p^{i})-1}}{2^{p^{n}-3}}.$
 \end{corollary}

As an immediate consequence of Corollary \ref{cor:mainref} and the above Corollary \ref{cor:snfforprime} one can obtain the Grothendieck group of $L_{K}(Pow^{\ast}(\mathbb{Z}_{p^n}))$, up to isomorphism, as follows.

\begin{theorem} \label{thm:Main_pgeq3}
     Let $p\geq 3$ be a prime. The Grothendieck group of $L_{K}(Pow^{\ast}(\mathbb{Z}_{p^n}))$ is isomorphic to 
     
     $$ \begin{cases} \underbrace{\mathbb{Z}_{2} \oplus \cdots \oplus \mathbb{Z}_{2}}_{(3^n-2n-1) \ \mathrm{times}} \oplus \underbrace{ \mathbb{Z}_{4} \oplus \cdots \oplus \mathbb{Z}_{4}}_{ (n-1) \ \mathrm {times}}\oplus \mathbb{Z}, & \text{when} \ p=3,\\ \\ \underbrace{\mathbb{Z}_{2} \oplus \cdots \oplus \mathbb{Z}_{2}}_{(p^n-2n-1) \ \mathrm{times}} \oplus \underbrace{ \mathbb{Z}_{4} \oplus \cdots \oplus \mathbb{Z}_{4}}_{ (n-1) \ \mathrm {times}}\oplus \mathbb{Z}_{|\eta|}, & \text{when} \ p\geq 5,
     \end{cases} $$
     where $\eta= \frac{\prod^{n}_{i=1} (2-\varphi(p^{i})) 2^{\varphi(p^{i})-1}}{2^{p^{n}-3}}$.
 \end{theorem}

\section{Proof of Theorem \ref{thm:gcd}}
\label{sect:proof_Thm}

This section is devoted to prove Theorem \ref{thm:gcd}. For completeness, we rename the matrix $I-A^{tr}_{Pow^{\ast}(\mathbb{Z}_{p^{n}})}$ as follows.

\begin{definition} \label{def:thematrix}
    Let $M(p^n)$ be the matrix
    \setlength{\arraycolsep}{.2pt} 
\renewcommand{\arraystretch}{1.3} 
\begin{equation*}
   M(p^n)=\begin{pNiceMatrix} 
    B_{p-1} & -1  &-1 & \ \cdots &  -1  \\
     0 &B_{p^2-p} & -1 &\ \cdots & -1\\
     0 & 0 &B_{p^3-p^2}  &\ \cdots & -1\\
    \vdots & \vdots & \vdots & \ddots & -1\\
    0& 0& 0& \ \cdots & \ \ B_{p^n-p^{n-1}}
    \end{pNiceMatrix} \end{equation*}
where $B_{k}$ is the circulant matrix associated to the vector $(1,-1,\ldots,-1) \in \mathbb{Z}^k$ for $ k \geq 2$. For completeness, we define $B_{1}=[1]$. We denote the greatest common divisor of the set of determinants of $k \times k$-minors of $M(p^n)$ by $d_{k}(M(p^n))$ where $1 \leq k \leq p^n-1$.  
\end{definition}

Every $(k\times k)$-submatrix $N$ of $M(p^n)$ is obtained by taking the intersection of rows $r_{i_{1}}, \dots, r_{i_k}$ and columns $c_{j_1}, \dots, c_{j_k}$ of $M(p^n)$ for some $1\leq i_1\leq \cdots \leq i_k \leq p^n-1$ and $1\leq j_1\leq \cdots \leq j_k \leq p^n-1$. In other words, $N$ is the matrix whose $(x,y)$-entry is the $(i_x,j_y)$-entry of $M(p^n)$. Let $R_{N}$ and $C_{N}$ be the sets of indices of rows and columns of $M(p^{n})$ forming $N$, respectively. These sets can be reindexed so that their elements are in increasing order:
\begin{eqnarray*}
R_N &=& \{\alpha^{1}_1, \ldots, \alpha^{1}_{k_{1}}, \ldots, \alpha^{n}_{1}, \ldots, \alpha^{n}_{k_{n}}\},\\
C_N &=& \{\beta^{1}_1, \ldots, \beta^{1}_{l_{1}}, \ldots, \beta^{n}_{1},\ldots,\beta^{n}_{l_{n}}\},
\end{eqnarray*}where
\[
0 \leq k_{i}, l_{i} \leq p^{i}-p^{i-1},
\]
\[
p^{i-1} \leq \alpha^{i}_j, \beta^{i}_{l} \leq p^{i}-1,
\]
for $1 \leq j\leq k_{i}$, $1 \leq l \leq l_{i}$, and $1 \leq i \leq n$. Here, upper indices represent the row blocks (or column blocks), in which the rows (or columns) are taken. We write $k_i = 0$ (or $l_i = 0$) to indicate that no rows (or columns) of the $i^{th}$ row block (or $i^{th}$ column block) is used to form $N$.  Since $N$ is a $(k \times k)$-matrix, we have
\[k_1 + \cdots + k_n = l_{1}+ \cdots l_{n}=k.\]

To calculate $d_{k}(M(p^n))$, we first find a large family of $(k \times k)$-submatrices of $M(p^n)$ consisting of singular matrices. For this, let $A^{i}_N = \{\alpha^{i}_1, \ldots, \alpha^{i}_{k_{i}}\}$ and $B^{i}_N = \{\beta^{i}_1, \ldots, \beta^{i}_{l_{i}}\}$ for a given $(k\times k)$-submatrix $N$ of $M(p^n)$.

\begin{lemma}
\label{lem:A_iB_i_singularity}
If $|A^{i}_N \setminus B^{i}_N| \geq 2$ or $|B^{i}_N \setminus A^{i}_N| \geq 2$ for some $1 \leq i \leq n$ then $\det N = 0$.
\end{lemma}

\begin{proof}
If there exists $j \neq j'$ such that  $\alpha^{i}_j, \alpha^{i}_{j'} \in A^{i}_N \setminus B^{i}_N$, i.e., $|A^{i}_N \setminus B^{i}_N| \geq 2$, then $(k_1 + \cdots + k_{i-1} + j)$-th and $(k_1 + \cdots + k_{i-1} + j')$-th rows of $N$ are both equal to the row vector
\[(\underbrace{0, \ldots, 0}_{k^{i-1} - 1}, -1, \cdots,-1).\] Therefore, the determinant of $N$ is zero. Similarly, if $|B^{i}_N \setminus A^{i}_N| \geq 2$, the matrix $N$ has $2$ equal columns and hence is singular.
\end{proof}

Let $A_{m+1}(x,y)$ denote the matrix obtained by inserting the row 
$(-1,\dots, -1) \in \mathbb{Z}^m $ to the $x^{th}$ position and then the column $(-1,\dots, -1)^t \in \mathbb{Z}^{m+1}$ to the $y^{th}$ position into $B_m$.
Clearly $A_{m+1}(x,y)$ is the square matrix of order $m+1$ with
\[
(A_{m+1}(x,y))_{u,v} = \begin{cases}
1,& \text{if } u=v > \max\{x,y\} \text{ or } u=v < \min\{x,y\}, \\
1,& \text{if } x < u = v+1 \leq y, \\
1,& \text{if } y \leq u = v-1 < x, \\
-1,& \text{otherwise},
\end{cases}
\] where $1 \leq u,v \leq m+1$.

Let $C_{m+1}(x)$ be the $(m+1) \times m$ matrix obtained by inserting the row $$(-1,\dots, -1) \in \mathbb{Z}^m $$ to the $x^{th}$ position into $B_m$, i.e.,
\[
(C_{m+1}(x))_{u,v} = \begin{cases}
1 & \text{if } u=v<x \text{ or } u=v+1>x, \\
-1 & \text{otherwise},
\end{cases}
\] where $1 \leq u \leq m+1$ and $1 \leq v \leq m$.
We denote the transpose of $C_{m+1}(x)$ by $D_m(x)$. In other words, $D_m(x)$ is the matrix obtained by inserting the column $(-1,\dots, -1)^t \in \mathbb{Z}^m $ to the $x^{th}$ position in $B_m$.

\begin{example} 
\[
A_4(2,3) = \begin{pNiceArray}{cccc}
1 & -1 & -1 & -1 \\
-1 & -1 & -1 & -1 \\
-1 & 1 & -1 & -1 \\
-1 & -1 & -1 & 1
\end{pNiceArray}
, \ \
C_4(2) = \begin{pNiceArray}{ccc}
1 & -1 & -1 \\
-1 & -1 & -1 \\
-1 & 1 & -1 \\
-1 & -1 & 1
\end{pNiceArray}, \ \ 
D_3(2) = \begin{pNiceArray}{cccc}
1 & -1 & -1 & -1 \\
-1 & -1 & 1 & -1 \\
-1 & -1 & -1 & 1
\end{pNiceArray}.
\]
\end{example}

\begin{corollary}
\label{cor:N_k's}
A non-singular $(k \times k)$-minor of $M(p^n)$ is of the form
\setlength{\arraycolsep}{.2pt} 
\renewcommand{\arraystretch}{1.3} 
\begin{equation}\label{mat:firstformofN}
   \begin{pNiceMatrix} 
    N_{k_1} & -1  &-1 & \ \cdots &  -1  \\
     0 &N_{k_2} & -1 &\ \cdots & -1\\
     0 & 0 &N_{k_3}  &\ \cdots & -1\\
    \vdots & \vdots & \vdots & \ddots & -1\\
    0& 0& 0& \ \cdots & \ \ N_{k_s}
    \end{pNiceMatrix} \end{equation}
where $N_{k_i}$ is one of the matrices $A_{k_i}(x_{i},y_{i})$, $B_{k_i}$, $C_{k_i}(x_{i})$, or $D_{k_i}(x_{i})$ for some $x_{i},y_{i}$ for all $1 \leq i \leq s$ where $1\leq k_i \leq p^{n-s+i}-p^{n-s+i-1}$ and $k_1 + \cdots + k_s = k$.
\end{corollary}

\begin{proof} Let $N$ be a $(k \times k)$-minor of $M(p^n)$ and $i_1,\cdots, i_s$ be such that $A_{N}^{i_t}\neq \emptyset$. Since $N$ is non-singular, $|A^{i_t}_N \setminus B^{i_t}_N| \leq 1$ and $|B^{i_t}_N \setminus A^{i_t}_N| \leq 1$ by Lemma \ref{lem:A_iB_i_singularity}. Let us denote the contribution of the $i_t$-th block of $M$ to the minor $N$ by $N_{k_t}$ where $k_t=|A^{i_t}|$. Note that every principal minors of $B_n$ is of the form $B_{n'}$ for some $n'\leq n$. If $A_{N}^{i_t}=B_{N}^{i_t}$ then $N_{k_t}$ is a principal minor of $B_{p^{i_t}-p^{i_t-1}}$ and hence $N_{k_t}=B_{k_t}$. If $A_{N}^{i_t}= B_{N}^{i_t} \cup \{x\}$ (disjoint union) then $N_{k_t}$ is a $k_t \times (k_t-1)$ matrix whose $x$-th row is the vector $(-1,\cdots,-1)$ and the matrix obtained by removing this row from $N_{k_t}$ is a principal minor of $B_{p^{i_t}-p^{i_t-1}}$. Hence $N_{k_t}=C_{k_t}$. The same argument applies to the other cases. Indeed, we have
\[ N_{k_t}=\begin{cases}
\setlength\arraycolsep{2pt}\begin{pNiceArray}{c|c|c}
0 & B_{k_{t}}  & -1
\end{pNiceArray}, & \text{if} \ A^{i_t}_N = B^{i_t}_N \neq \emptyset \\
\setlength\arraycolsep{2pt}\begin{pNiceArray}{c|c|c}
0 & C_{k_{t}}(x) & -1
\end{pNiceArray}, & \text{if} \ A^{i_t}_N = B^{i_t}_N \cup \{x\}, \\
\setlength\arraycolsep{2pt}\begin{pNiceArray}{c|c|c}
0 & D_{k_{t}}(x) & -1
\end{pNiceArray}, & \text{if} \ B^{i_t}_N = A^{i_t}_N \cup \{x\}, \\
\setlength\arraycolsep{2pt}\begin{pNiceArray}{c|c|c}
0 & A_{k_{t}}(x,y) & -1
\end{pNiceArray}, & \text{if} \ A^{i_t}_N \setminus B^{i_t}_N = \{x\} \ \text{and} \ B^{i_t}_N \setminus A^{i_t}_N = \{y\}.
\end{cases} \] where the unions in the second and third cases are disjoint.
\end{proof}

Let us denote the $(k \times k)$ square matrix of the form (\ref{mat:firstformofN}) with
\[
N_{k_i} = \begin{cases}
A_{k_i}(1,1), & \text{if } j_i = 0, \\
B_{k_i}, & \text{if } j_i = 1, \\
C_{k_i}(1), & \text{if } j_i = 2, \\
D_{k_i}(1), & \text{if } j_i = 3.
\end{cases}
\] by $N^{i}_k[k_1^{j_1}, \ldots, k_s^{j_s}]$ for $1 \leq i \leq s \leq n$. Since $N^{i}_k[k_1^{j_1}, \ldots, k_s^{j_s}]$ is a square matrix of order $k$, we have $k_1+\cdots +k_s=k$ and
\[
|\{j_i : j_i = 2\}| = |\{j_i : j_i = 3\}|.
\]
As any matrix of the form (1) can be transformed into $N^{i}_k[k_1^{j_1}, \ldots, k_s^{j_s}]$ by permuting its rows and/or columns, we have the following result.
\begin{corollary}\label{cor:detN_k}
The number $d_k(M(p^n))$ is equal to the greatest common divisor of determinants of the matrices of the form $N^{i}_k[k_1^{j_1}, \ldots, k_s^{j_s}]$ where $1\leq k_i \leq p^{n-s+i}-p^{n-s+i-1}$ for some $s\leq n$ provided that $j_i = 1$ whenever $k_i=p^{n-s+i}-p^{n-s+i-1}$.
\end{corollary}
 
We denote the submatrix 
\setlength{\arraycolsep}{.2pt} 
\begin{equation*} \label{eq:upperblockform}\NiceMatrixOptions{xdots/shorten=0.5em}
\renewcommand{\arraystretch}{1.3} 
   \begin{pNiceMatrix} 
    N_{k_1} & -1  &-1 & \ \cdots &  -1  \\
     0 &N_{k_2} & -1 &\ \cdots & -1\\
     0 & 0 &N_{k_3}  &\ \cdots & -1\\
    \vdots & \vdots & \vdots & \ddots & -1\\
    0& 0& 0& \ \cdots & \ \ N_{k_i}
    \end{pNiceMatrix} \end{equation*} of $N^{i}_k[k_1^{j_1}, \ldots, k_s^{j_s}]$ by $N_k^{i}[k_1^{j_1}, \ldots, k_s^{j_s}]$ and the submatrix
\setlength{\arraycolsep}{.2pt} 
\renewcommand{\arraystretch}{1.3} 
\[\NiceMatrixOptions{xdots/shorten=0.5em}
   \begin{pNiceMatrix} 
    N_{k_{i+1}} & -1  &-1 & \ \cdots &  -1  \\
     0 &N_{k_{i+2}} & -1 &\ \cdots & -1\\
     0 & 0 &N_{k_{i+3}}  &\ \cdots & -1\\
    \vdots & \vdots & \vdots & \ddots & -1\\
    0& 0& 0& \ \cdots & \ \ N_{k_s}
    \end{pNiceMatrix} \] 
by $\overline{N}^{i}_{k}[k_1^{j_1}, \ldots, k_s^{j_s}]$. We denote the sums $\sum\limits_{l=1}^i k_l$ and $ \sum\limits_{l=1}^i   \dfrac{(-1)^{j_l+1}j_l(j_l-1)}{j_{l}!}$ by $m_i$ and $q_i$, respectively. Note that the matrix $N_k^{i}[k_1^{j_1}, \ldots, k_i^{j_i}]$ is a $(m_i\times n_i)$-matrix where $n_i=m_i+q_i$.

\begin{theorem}
 \label{thm:singular_forms_forq_i}   
The matrix $N_k[k_1^{j_1}, \ldots, k_s^{j_s}]$ is singular if one of the following conditions holds for some $1 \leq i \leq s $:

\begin{itemize}
    \item[1)]  $q_{i-1}=0$, $j_i = 1$ and $k_i = 2$,
 \item[2)]  $q_{i-1}= 0$ and $j_i = 3$,
\item[3)]  $q_{i}>0$,
\item[4)]  $j_i = 2$ and $j_{i+1} \in \{0,2\}$,
\item[5)]  $j_{i+1} = 3$ and $j_i \in \{0,3\}$.
\end{itemize}
\end{theorem}

\begin{proof} Write $N_k^i = N^{i}_k[k_1^{j_1}, \ldots, k_s^{j_s}]$ and $\overline{N}_k^i = \overline{N}^{i}_k[k_1^{j_1}, \ldots, k_s^{j_s}]$ for short. With the abuse of notation, we denote the $i$-th column and $i$-th row of the matrix in question by $c_i$ and $r_i$, respectively provided that it is clear from the context. When $q_{i-1}=0$, $N_k^{i-1}$ and $\overline{N}^{i-1}_{k}$ are square matrices. Furthermore, if $q_{i-1}=0$ and $j_i=1$ then
$$\det (\overline{N}^{i-1}_{k}) = (\det B_{2} )\cdot (\det \overline{N}_k^{i})=0$$ as $\det B_{2}=0.$ Therefore, $N_k^s$ is singular. 

Now consider the case where $q_{i-1}=0$, $j_i=3$. The matrix $\overline{N}^{i-1}_{k}$ has the following form
\setlength\arraycolsep{2pt}
\def\arraystretch{2}
\[\begin{pNiceArray}{c|c|c}
\Block{}{%
\begin{matrix}
-1 \\
\vdots \\
-1
\end{matrix}} & \Block{}<\Large> {B_{k_i}} & \Block{}<\Large>{-1} \\
\hline
\Block{1-2}<\Large> 0 &  & \Block{}<\Large>{\overline{N}_k^i} 
\end{pNiceArray}.\]
Since the linear combination $(2-k_i)c_{1}+c_{2}+\dots+c_{k_{i}+1}$ of its columns is the zero vector, the determinant of $\overline{N}^{i-1}_{k}$ is zero and hence $N_k^s$ is a singular matrix.

When $q_{i-1}>0$,  $N_k^s$ can be written as $\begin{pNiceArray}{c|c} \ S \ & U \\ \hline \ 0 \ &  T\end{pNiceArray}$ for some matrix U, where $S$ is the matrix obtained by adding $q_i$ zero rows at the end of $N_k^i$ and $T$ is the matrix obtained by removing the first $q_i$ rows of $\overline{N}_k^i$. Since $S$ is singular and both $S$ and $T$ are square matrices, $N_k^s$ is singular.

When $j_{i}=2$ and $j_{i+1} \in \{0,2\}$, $N_k^s$ is  one of the following matrices

\setlength{\arraycolsep}{4pt} 
\renewcommand{\arraystretch}{2} 
\[
   \begin{pNiceMatrix} 
    N^{i-1}_{k} & -1   &  -1&  -1  \\
     0 &C_{k_{i}}  & -1&  -1\\
     0 & 0 &A_{k_{i+1}}  & -1& \\
    0& 0& 0&   \overline{N}^{i+1}_{k}
    \end{pNiceMatrix} 
    \quad \text{or} \quad
       \begin{pNiceMatrix} 
    N^{i-1}_{k} & -1   &  -1&  -1  \\
     0 &C_{k_{i}}  & -1&  -1\\
     0 & 0 &C_{k_{i+1}}  & -1& \\
    0& 0& 0&   \overline{N}^{i+1}_{k}
    \end{pNiceMatrix} 
    \]

In both cases, the following linear combination of rows
\[(k_i-3) r_{m_{i-1}+1}- \Big(r_{m_{i-1}+2}+r_{m_{i-1}+3}+\dots+r_{m_{i}}\Big)+2r_{m_{i+1}}   \]
is zero and hence the matrix $N_k^s$ is singular.

When $j_{i+1}=3$ and $j_i \in \{0,3\}$, the matrix $N_k^s$ is one of the following matrices
\setlength\arraycolsep{2pt} 
\def\arraystretch{2}
\[\begin{pNiceArray}{c|c|c|c}[margin]
\Block{}<\large>{N_k^{i-1}} & \Block{}<\large>{-1} & \Block{}<\large>{-1} & \Block{}<\large>{-1} \\
\hline
\Block{}<\large>{0} & \Block{1-1}<\large>{\begin{array}{c|c}
    \begin{matrix}
-1 \\
\vdots \\
-1
\end{matrix} & B_{k_{i}}
\end{array}} & \Block{}<\large>{-1} & \Block{}<\large>{-1} \\
\hline
\Block{}<\large>{0} & \Block{}<\large>{0} & \Block{1-1}<\large>{\begin{array}{c|c}
\begin{matrix}
-1 \\
\vdots \\
-1
\end{matrix} & B_{k_{i+1}}
\end{array}} & \Block{}<\large>{-1} \\
\hline
\Block{}<\large>{0} & \Block{}<\large>{0} & \Block{}<\large>{0} & \Block{}<\large>{\overline{N}_{k}^{i+1}}.
\end{pNiceArray} \ \ \mathrm{or} \ \ \begin{pNiceArray}{c|c|c|c}[margin]
\Block{}<\large>{N_k^{i-1}} & \Block{}<\large>{-1} & \Block{}<\large>{-1} & \Block{}<\large>{-1} \\
\hline
\Block{}<\large>{0} & \Block{1-1}<\large>{\begin{array}{c|c}
    -1&-1 \cdots -1 \\
    \hline
    \begin{matrix}
-1 \\
\vdots \\
-1
\end{matrix} & B_{k_{i}}
\end{array}} & \Block{}<\large>{-1} & \Block{}<\large>{-1} \\
\hline
\Block{}<\large>{0} & \Block{}<\large>{0} & \Block{1-1}<\large>{\begin{array}{c|c}
    \begin{matrix}
-1 \\
\vdots \\
-1
\end{matrix} & B_{k_{i+1}}
\end{array}} & \Block{}<\large>{-1} \\
\hline
\Block{}<\large>{0} & \Block{}<\large>{0} & \Block{}<\large>{0} & \Block{}<\large>{\overline{N}_{k}^{i+1}}
\end{pNiceArray}.\] In both cases the sum of the column vectors 
\[2 c_{n_{i-1}+1}+ (k_{i+1}-2) c_{n_i+1}-\Big(c_{n_{i}+2}+\cdots+c_{n_{i+1}-1}\Big)\]
is zero. Thus the matrix $N_k^s$ is singular.\end{proof}

For the following observations, we write $N=N^{i}_k[k_1^{j_1}, \ldots, k_s^{j_s}]$ for short.
\begin{proposition}\label{prop:det_for(2,3)} If $j_i=2$ and $j_{i+1}=3$ then
    $$\det N =  (-1)^{k_i-1}\det N_k[k_1^{j_1},\ldots,k_{i-1}^{j_{i-1}},(k_i+k_{i+1})^0,k_{i+2}^{j_{i+2}},\ldots,k_s^{j_s}].$$
\end{proposition}

\begin{proof}
When $j_i=2$ and $j_{i+1}=3$, the matrix $N$ has the following form
\setlength\arraycolsep{2pt} 
\def\arraystretch{2}
\[ 
\begin{pNiceArray}{c|c|c}[margin]
N_k^{i-1} & -1 & -1  \\
\hline
0 & \begin{array}{c|c}
    C_{k_{i}} & -1 \\
    \hline
    0 & D_{k_{i+1}} \\
\end{array} & -1  \\
\hline
0 & 0 & \overline{N}_{k^{i+1}}
\end{pNiceArray}.
\]
By permuting the columns by permutation $(n_{i-1}+1, n_{i-1}+2, \cdots,n_{i-1}+k_i)$, we obtain the matrix
\setlength\arraycolsep{2pt} 
\def\arraystretch{2}
\[
\begin{pNiceArray}{c|c|c}
N_k^{i-1} & -1 & -1  \\
\hline
0 & \begin{array}{c|c}
    -1 & -1  \quad  \cdots \quad -1  \\
    \hline
    \begin{matrix}-1 \\ \vdots \\ -1 \end{matrix} & \begin{array}{c|c}
            B_{k_{i}-1} & -1 \\
            \hline
            0 & B_{k_{i+1}}
        \end{array} \\
\end{array} & -1  \\
\hline
0 & 0 & \overline{N}^{i+1}_{k}
\end{pNiceArray}.
\] whose determinant is $(-1)^{k_i-1}$ times that of $N$. By adding the linear combination $\frac{1}{2}\big((k_i-1)r_{m_{i-1}+1} - \sum_{t=2}^{k_i} r_{m_{i-1}+t}\big)$ of rows of this matrix to the $({m_{i-1}+k_i+j})$-th row of it for $1\leq j \leq k_{i+1}$, we obtain the matrix \\ 
 $N^{i}_k[k_1^{j_1},\ldots,k_{i-1}^{j_{i-1}},(k_i+k_{i+1})^0,k_{i+2}^{j_{i+2}},\ldots,k_s^{j_s}]$, as desired.
\end{proof}

\begin{proposition}\label{det_for(2,1)} If $j_i=2$, $j_{i+1}=1$ and $k_i \neq 3$ then
$$\det N =  \alpha \cdot \det N_k[k_1^{j_1},\ldots,k_{i-1}^{j_{i-1}},k_i^1,k_{i+1}^2,k_{i+2}^{j_{i+2}},\ldots,k_s^{j_s}]$$
where $\alpha = (-1)^{k_i} \cdot \frac{2}{3-k_i}$.
\end{proposition}
\begin{proof} The matrix $ N_k[k_1^{j_1},\ldots,k_{i-1}^{j_{i-1}},k_i^2,k_{i+1}^1,k_{i+2}^{j_{i+2}},\ldots,k_s^{j_s}]$ is of the following form
\setlength\arraycolsep{2pt} 
\def\arraystretch{2}
\[
\begin{pNiceArray}{c|c|c}[margin]
N_k^{i-1} & -1 & -1  \\
\hline
0 & \begin{array}{c|c}
    C_{k_{i}} & -1 \\
    \hline
    0 & B_{k_{i+1}} \\
\end{array} & -1  \\
\hline
0 & 0 & \overline{N}^{i+1}_{k}
\end{pNiceArray}.
\] Permuting the columns of this matrix by the permutation \\
$(m_{i-1}+2,m_{i-1}+1, m_{i-1}+k_{i}, m_{i-1}+k_{i-1}, \dots, m_{i-1}+3$) results in the matrix
\[
\begin{pNiceArray}{c|c|c}[margin]
N_k^{i-1} & -1 & -1  \\
\hline
0 & \begin{array}{c|c}
    B_{k_{i}-1} & -1 \\
    \hline
    -1  \quad  \cdots \quad -1&-1  \quad  \cdots \quad -1\\
    \hline
    0 & B_{k_{i+1}} \\
\end{array} & -1  \\
\hline
0 & 0 & \overline{N}^{i+1}_{k}
\end{pNiceArray}.\] By interchanging the row $r_{m_{i-1}+k_{i}}$ of the above matrix with the linear combination $r_{m_{i-1}+k_{i}}+\frac{1}{3-k_{i}}(r_{m_{i-1}+1}+r_{m_{i-1}+2}+\dots+r_{m_{i-1}+k_{i}-1})$ of its rows we obtain the matrix
     \[
    \begin{pNiceArray}{c|c|c}
    N^{i-1}_{k}&-1&-1\\
    \hline
    0&  \begin{array}{c|c|c}   B_{k_{i}-1} &
        -1 &-1  \\
        \hline
         0&\frac{2}{k_{i-3}} \quad \cdots \quad \frac{2}{k_{i-3}}& \frac{2}{k_{i-3}} \quad \cdots \quad \frac{2}{k_{i-3}}\\
         \hline
        0& B_{k_{i+1}}&-1
    \end{array} \\
    \hline
    0&0& \overline{N}_{k}^{i+1}
    \end{pNiceArray}.
    \]
Since its determinant is $-\frac{2}{k_i-3}$ times the determinant of \\
$N^{i}_k[k_1^{j_1}, \ldots, k_{i-1}^{j_{i-1}}, (k_i-1)^1, (k_{i+1}+1)^2, \ldots, k_s^{j_s}]$, the result follows.
\end{proof}

\begin{proposition}\label{prop:det_for(2,1)to(0,2)} If $j_i=2$ and $j_{i+1}=1$ then \[ \det N = \pm 2 \det N_k[k_1^{j_1}, \ldots, k_{i-1}^{j_{i-1}}, k_i^{0}, k_{i+1}^{2}, k^{j_{i+2}}_{i+2}, \ldots, k_s^{j_s}]\]
\end{proposition}

\begin{proof}
    We prove by induction on $k_{i+1}$. For $k_{i+1}=2$, $N_k[k_1^{j_1}, \ldots, k_{i-1}^{j_{i-1}}, k_i^{2}, k_{i+1}^{1}, k^{j_{i+2}}_{i+2}, \ldots, k_s^{j_s}]$ is of the form
\def\arraystretch{1.5}
    \[
    \begin{pNiceArray}{c|c|c}[margin]
N_k^{i-1} & -1 & -1  \\
\hline
0 & \begin{array}{c|c}
    C_{k_{i}} & \begin{array}{cc}
        -1 &-1  
    \end{array}\\
    \hline
    0 & \begin{array}{cc}
       \ 1 &-1  \\
         -1& \ 1 
    \end{array} \\
\end{array} & -1  \\
\hline
0 & 0 & \overline{N}^{i+1}_{k}
\end{pNiceArray}.
    \]
Replacing the column $c_{n_{i}+1}$ by $c_{n_{i}+1}+c_{n_{i}+2}$, we obtain the matrix 
\renewcommand{\arraystretch}{1.3} 
\[
\begin{pNiceArray}{c|c|c}[margin]
N^{i-1}_{k} & -1 \vline \begin{array}{cc}
    -2 & -1 \\
    \vdots & \vdots \\
    -2 & -1
\end{array} &-1\\
\hline 
0 & \begin{array}{c|c}
    C_{k_{i}} & \begin{array}{cc}
        -2 & -1 \\
        \vdots&\vdots\\
        -2 & -1 
    \end{array}  \\
    \hline
    0 & \begin{array}{cc}
        0 & -1 \\
         0 & 1 
    \end{array} 
\end{array} & -1 \\
\hline
0 & 0 & \overline{N}_{k}^{i+1}
\end{pNiceArray}.
\] Since the determinant of the above matrix is equal to $\pm 2$ times the determinant of the matrix 
 \[
\begin{pNiceArray}{c|c|c}[margin]
N_k^{i-1} & -1  & -1  \\
\hline
0 & \begin{array}{c|c}
    A_{k_{i}}(1,k_i) & -1 \\
    \hline
    0 & C_2 \\
\end{array} & -1  \\
\hline
0 & 0 & \overline{N}^{i+1}_{k}
\end{pNiceArray},
\]
the base case of the induction is proved. 

Assume that the statement holds for all $k_{i+1} \leq t.$ Let $k_{i+1}=t+1$. Subtracting the row $r_{m_{i}}+2$ of $N_k[k_1^{j_1}, \ldots, k_{i-1}^{j_{i-1}}, k_i^{2}, k_{i+1}^{1}, k^{j_{i+2}}_{i+2}, \ldots, k_s^{j_s}]$ from the row $r_{m_{i}+1}$ results in the vector $(0,\cdots,0, \underbrace{2}_{n_i+1},-2,0,\cdots,0).$ By taking the cofactor expansion along this row, one can compute its determinant as
$$(-1)^{m_{i}+n_{i}+2} \ 2 \ \mathrm{det}N_k[k_1^{j_1}, \ldots, k_{i-1}^{j_{i-1}}, k_i^{2}, (k_{i+1}-1)^{1}, k^{j_{i+2}}_{i+2}, \ldots, k_s^{j_s}]$$ by Theorem \ref{thm:singular_forms_forq_i}. On the other hand, applying the same operations to the matrix $N_k[k_1^{j_1}, \ldots, k_{i-1}^{j_{i-1}}, k_i^{0}, k_{i+1}^{2}, k^{j_{i+2}}_{i+2}, \ldots, k_s^{j_s}]$ yields a determinant of
$$(-1)^{m_{i}+n_{i}+2} \ 2 \ \mathrm{det} N_k[k_1^{j_1}, \ldots, k_{i-1}^{j_{i-1}}, k_i^{0}, (k_{i+1}-1)^{2}, k^{j_{i+2}}_{i+2}, \ldots, k_s^{j_s}].$$ The result follows by induction.

\end{proof}
Let $\mathcal{M}_{k}$ be the set of $(k\times k)$-matrices of the form
$N_k[k_1^{j_1}, \ldots, k_s^{j_s}]$ where $j_{i} \in \{0,1\}, 1\leq k_{i} \leq p^{n+i-s}-p^{n+i-s-1}$ provided that $j_{i}=1$ whenever $k_{i} =p^{n+i-s}-p^{n+i-s-1}$. 

\begin{theorem} For any $1\leq k \leq p^n-1$, $d_k(M(p^n))$ is equal to the greatest common divisor of the determinants of matrices in $\mathcal{M}_{k}$.
\end{theorem}

\begin{proof} Let $N=N_k[k_1^{j_1}, \ldots, k_s^{j_s}]$ be a non-singular $(k \times k)$ submatrix of $M(p^n)$ for which $j_i=2$ for some $1\leq i \leq s$. Let $t_1$ be the smallest $t$ for which $j_t=2$. Since $N$ is non-singular, $q_i \leq 0$ by Theorem \ref{thm:singular_forms_forq_i}. Therefore $j_i \neq 3$ for any $i<t_1$ and hence $j_i \in \{0,1\}$ for $i<t_1$. If $j_{t_1}+1=3$, we can remove the $t_1$-th and $(t_1+1)$-th row blocks and replace them with a row block of the form $(0|A_{k_{t_1}+k_{t_1+1}})$ by Proposition \ref{prop:det_for(2,3)}. This might only change the sign of the determinant and, hence, does not affect the greatest common divisor. Therefore, we can assume that whenever $j_i=2$ then $j_{i+1}=1$ by Theorem \ref{thm:singular_forms_forq_i}. On the other hand, if $j_i=2, \ j_{i+1}=1 $ and $j_{i+2}\in\{0,2\}$ then the determinant is zero by Proposition \ref{prop:det_for(2,1)to(0,2)} and Theorem \ref{thm:singular_forms_forq_i}. So $j_i+2 \in\{1,3\}$. If $j_{i+2}=1$, applying proposition \ref{prop:det_for(2,1)to(0,2)} twice, followed by Theorem \ref{thm:singular_forms_forq_i}, yields $j_{i+3}\in \{1,3\}$. A similar argument shows that if $t_1'$ is the first block for which $j_{t_1'}=3$, then for $t_1 < i< t_1'$, $j_i=1$. Since all the blocks of $N_k^{t_1'}$ are square blocks except the $t_1$-th and $t_1'$-th, which are $C_{k_{t_1}}$ and $D_{k_{t_1'}}$, respectively, $N_k^{t_1'}$ is a square matrix. Hence if $t_2$ and $t_2'$ are the next blocks for which $j_{t_2}=2$ and $j_{t_2'}=3$ then $t_2<t_2'$ by Theorem \ref{thm:singular_forms_forq_i}. We may similarly assume that there exists a $u$ and $t_1< t_1'< t_2 < t_2'<\cdots <t_u <t_u'$ such that 
\[
j_v = \begin{cases}
2, & \text{if } v \in \{t_1,\cdots, t_l\} , \\
3, & \text{if } v \in \{t_1',\cdots, t_l'\}, \\
1, & \text{if } t_r<v<t_r' \ \text{where} \ 1\leq r \leq u, \\
0 \ \text{or} \ 1, & \text{otherwise}.
\end{cases}
\]

By Proposition \ref{prop:det_for(2,3)} and Proposition \ref{prop:det_for(2,1)to(0,2)}, the determinant of $N$ is equal to the $\pm 2^{\overset{u}{ \underset{r=1}{\sum}}(t_r'-t_r-1)}$ times the determinant of the $N'=N_k[l_1^{i_1}, \ldots, l_{s-u}^{i_{s-u}}]$ where
 \[
i_v = \begin{cases}
0, & \text{if } t_r-r+1\leq v \leq t_r'-r \ \text{where} \ 1\leq r \leq u, \\
j_{v}, & \text{if} \ v< t_1 \ \mathrm{or} \ v>t_u'-u,\\
j_{v+r}& \text{if} \ t_r'-r<v\leq t_{r+1}-r-1, \ 1 \leq r \leq u-1. 
\end{cases}, \] and \[ l_v = \begin{cases}
k_{v+r-1}, & \text{if } \ t_r-r+1\leq v \leq t_r'-r-1 \ \text{where} \ 1\leq r \leq u, \\
k_{t_r'-1}+k_{t_r'}, & \text{if } \ v=t_r'-r\ \text{where} \ 1\leq r \leq u, \\
k_{v}, & \text{if} \ v< t_1 \ \mathrm{or} \ v>t_u'-u\\
k_{v+r}& \text{if} \ t_r'-r<v\leq t_{r+1}'-r-1, \ 1 \leq r \leq u-1. 
\end{cases}
\]
Here, $N'$ is not necessarily a submatrix of $M(p^n)$ as $k_{t_r'-1}+k_{t_r'}$ can be too big to fit in a block of $M(p^n)$ by Corollary \ref{cor:detN_k}. However, the determinant of $A_{t}(1,1)$ is $-2^{t-1}$ (see the proof of Lemma 3.1 in \cite{das2023grothendieck}) and the blocks of $N'$ are all squares. Therefore its determinant is equal to $2^u$ times the determinant of $N''=N[k_1^{r_1},\cdots,k_s^{r_s}]$ where  
\[r_v = \begin{cases}
0, & \text{if } t_r\leq v \leq t_r' \ \text{where} \ 1\leq r \leq u, \\
j_{v}, & \text{otherwise}  
\end{cases} \] which clearly, is a $(k \times k)$ submatrix of $M(p^n)$. Since the determinant of $N$ is equal to $\pm 2^{\overset{u}{ \underset{r=1}{\sum}}(t_r'-t_r)}$ times $ \det (N'')$, the greatest common divisor of the determinants of $N$ and $N''$ is equal to the determinant of $N''$. \end{proof}
Now we are ready to prove Theorem \ref{thm:gcd}.

\begin{proofofgcd} By above theorem, it suffices to find the greatest common divisors of the determinants of elements of $\mathcal{M}_{k}$. Since $\det B_{n}=(2-n)\cdot 2^{n-1}$ and $\det A_{n}(1,1)=-2^{n-1}$ (See proof of the Lemma 3.1 in \cite{das2023grothendieck}), the determinant of the matrix $N^{i}_k[k_1^{j_1}, \ldots, k_s^{j_s}]$ in $\mathcal{M}_{k}$ is equal to
\begin{equation}\label{eq:detofAB} (-1)^{s-j} 2^{k-s} \prod_{j_l \neq 0} (2-k_{l} )\end{equation}
where $j=\sum^{s}_{i=1}j_{i}$. Now we calculate the numbers $d_k(M(p^n))$ case by case.

If $1\leq k \leq n$, $N^{i}_{k}[1^{1},1^{1},\dots,1^{1}] \in \mathcal{M}_{k}$. Since $\det N_{k}[1^{1},1^{1},\dots,1^{1}]=1$, $d_k(M_p)=1$ in this case.

For $n+1\leq k \leq p^n-n-1, k=\sum k_{i}>s $ and hence $2^{k-n}$ divides the determinant of all $(k \times k)$ minors as $2^{k-n} \mid 2^{k-s}$. On the other hand, there exists $k_{1},k_{2},\dots,k_{n}$ with $1\leq k_{i} < p^{i}-p^{i-1}$ such that $\sum k_{i}=k$. Therefore $N_k[k_{1}^{0},k_{2}^{0},\dots,k_{n}^{0}]$ is an element of $\mathcal{M}_{k}$. Since 
     $$ \det N_{k}[k_{1}^{0},k_{2}^{0},\dots,k_{n}^{0}]=\pm 2^{k-n} $$
     $d_k(M(p^n))=2^{k-n}$ for $n+1 \leq k \leq p^{n}-n-1$.
     
To show that $d_k(M(p^n))=2^{2k-p^{n}+1}$ for $p^{n}-n \leq k \leq p^{n}-2$, let $k=p^{n}-m$ for some $2 \leq m \leq n$ and $N=N_{k}[k_1^{j_1}, \ldots, k_s^{j_s}] \in \mathcal{M}_{k}$. We first show that $2^{p^n-2m+1}$ divides $\det N$. By the equation \ref{eq:detofAB}, $2^{p^n-m-s}$ divides the determinant. Therefore, if $s+1 \leq m$ then $2^{p^n -2m+1}$ divides $2^{p^n-m-s}$ and hence the determinant of $N$. To prove other cases, we show that $2^{s-m+1}$ divides $\underset{j_l \neq 0}{\prod} (2-k_{l} )$. For this, it suffices to show that at least $(s-m+1)$ blocks of $M(p^n)$ contribute as a whole to form $N$ as $2$ divides $2-p^j+p^{j-1}$ for any $j>1$. Note that $N$ is obtained from $M(p^n)$ by removing $(n-s)$ block completely. By removing such blocks, we remove at least 
\[p^{n-s}-1=p-1+p^{2}-p+\dots+p^{n-s}-p^{n-s-1} \] many rows. So to obtain $N_{k}[k_1^{j_1}, \ldots, k_s^{j_s}]$, at most $m-p^{n-s}$ many rows must be removed from the remaining blocks of $M$. In other words, we need to take at least $s-m+p^{n-s}$ many of remaining blocks as a whole to obtain $N$. So $2^{p^{n}-2m+1}$ divides $\det N$ for any $N$ in $\mathcal{M}_k$ and hence it divides $d_k(M(p^n))$ for $p^{n}-n \leq k \leq p^{n}-2$. To show that $2^{p^{n}-2m+1}$ is equal to $d_k(M(p^n))$, we construct a subset of $\mathcal{M}_{k}$ whose elements' determinants have greatest common divisor $2^{p^{n}-2m+1}$. Let $\beta$ be the subset of $\mathcal{M}_{k}$ which consists of $N^{i}_k[k_1^{j_1},\cdots,k_n^{j_n}]$ where
     $$(j_i,k_i)=\begin{cases}
        (1, p^i-p^{i-1}), & \text{if} \ i \neq \theta_r \ \text{for \ any}\ 1\leq r \leq m-1, \\
        (0,p^{\theta_r}-p^{\theta_r-1}-1) , & \text{if} \ i=\theta_r \ \text{for \ some}\ 1\leq r \leq m-1. \end{cases}$$
    for some $1 \leq \theta_{1}<\theta_{2}<\dots<\theta_{m-1}\leq n$. In other words, $\beta$ consists of $(k\times k)$-minors obtained by removing the rows $\{ p^{\theta_{1}-1}, \dots,p^{\theta_{m-1}-1}\}$ and the columns $\{ p^{\theta_{1}-1}+1, \dots,p^{\theta_{m-1}-1}+1\}$ from $M$ where $1 \leq \theta_{1}<\theta_{2}<\dots<\theta_{m-1}\leq n$. Let us denote such an element of $\beta$ by  $N\{\theta_{1},\dots,\theta_{m-1}\}$ for short. Then 
     \[ \det N\{\theta_{1},\dots,\theta_{m-1}\}=\pm 2^{p^{n}-m-n} \prod_{1\leq i_{1}< \dots <i_{n-m+1} \leq n}(2-p^{i_{t}}+p^{i_{t}-1})\]
     where $\{i_{1},i_{2},\dots,i_{n-m+1}\} =\{ \theta_{1}, \theta_{2},\dots,\theta_{m-1}  \}^{c} $. Clearly, the greatest common divisor of such determinants is $2^{p^n-2m+1}$. Thus $d_k(M(p^n))$ is equal to $2^{2k-p^n+1}$ for $k=p^{n}-m$ with $2 \leq m \leq n$.

Since
     $$
        \det M= \prod^{n}_{i=1} \det B_{p^{i}-p^{i-1}}= \begin{cases} \prod^{n}_{i=1} (2-\varphi(p^{i})) 2^{\varphi(p^{i})-1}, & \text{when} \ p\geq 5,\\
        0, & \text{when} \ p=3,
        \end{cases}
     $$
we have
\[d_k(M_p)=\begin{cases}
       |\prod^{n}_{i=1} (2-\varphi(p^{i})) 2^{\varphi(p^{i})-1}|, & \text{for}\ k=p^n-1 \ \text {and} \ p\geq 5,\\
       0, & \text{for}\ k=p^n-1 \ \text {and} \ p=3.
    \end{cases}
   \]
 \end{proofofgcd}

\section{Grothendieck group of $L_K(Pow^{\ast}(\mathbb{Z}_{2^n }))$}
\label{sec:pow2n}

As in the previous section, we denote the vertex corresponding to $a$ in $Pow^{\ast}(\mathbb{Z}_{2^{n}})$ by $v_a$ for any positive integer $1\leq a \leq 2^n-1$. Then the vertex set of $Pow^{\ast}(\mathbb{Z}_{2^{n}})$ is $\{v_1,\dots,v_{p^n-1}\}$ and there is an edge from $v_a$ to $v_b$ if and only if $a\neq b$ and $\mathrm{gcd} \{a,p^n\}$ divides $\mathrm{gcd} \{b,p^n\}$. Therefore $v_{2^{n-1}}$ is a sink and by reordering the vertices, the adjacency matrix of $Pow^{\ast}(\mathbb{Z}_{p^{n}})$ can be written in the following form

\setlength{\arraycolsep}{.2pt} 
\renewcommand{\arraystretch}{1.3} 
\[\NiceMatrixOptions{xdots/shorten=0.5em}
\begin{pNiceMatrix}[columns-width=0.8cm]
0 & 0 &0 & \ \cdots &  0  \\
   1& S_{2}  &0 & \ \cdots &  0  \\
     1&1 &S_{4} &\ \cdots & 0\\
    1&\vdots & \vdots & \ddots  & 0\\
    1& 1& 1& \ \cdots & \ \ S_{2^{n-1}}
    \end{pNiceMatrix}\] 
where the first block consists of matrices of dimensions $1\times 2^{i}$ for $0 \leq i \leq n-1$. Therefore, the matrix $(I_{ns} - {A_{Pow^{\ast}(\mathbb{Z}_{2^{n}})}}_{ns})^{tr}$ is of the following $(2^n-1) \times (2^{n}-2)$ matrix 
\[ \begin{pNiceMatrix}[columns-width=0.8cm]
-1 & -1 &-1 & \ \cdots &  -1  \\
    B_{2} & 0 &0 & \ \cdots &  0  \\
     1 &B_{4} & 0 &\ \cdots & 0\\
    \vdots & \vdots & \vdots & \ddots & 0\\
    1& 1& 1& \ \cdots & \ \ B_{2^{n-1}}
    \end{pNiceMatrix}.\]
\begin{example} Let $G$ be a cyclic group of order $8$ generated by $a$. When we order the vertices in the given order $v_{a^4},v_{a^2},v_{a^6}, v_{a},v_{a^3},v_{a^5},v_{a^7}$, the adjcency matrix of $Pow^{\ast}(\mathbb{Z}_{8})$ takes the form
\[\begin{pNiceMatrix}[columns-width=0.8cm]
0 & 0 &0 &0 &0&0 &0   \\
   1& 0&1  &0 &0&0  &  0  \\
     1&1 &0&0 &0&0  &  0  \\
    1&1 &1 &0 &1&1  &  1  \\
    1&1 &1 &1 &0&1  &  1  \\
    1&1 &1 &1 &1&0  &  1  \\
    1&1 &1 &1 &1&1  &  0  \\
    \end{pNiceMatrix}\] 
and hence the matrix $I_{ns} - (A_{Pow^{\ast}(\mathbb{Z}_{8})})^{tr}_{ns}$ is  
\[\begin{pNiceMatrix}[columns-width=0.8cm]
-1 & -1 &-1 &-1 &-1&-1    \\
   1& -1&-1  &-1 &-1&-1    \\
     -1&1 &-1&-1&-1&-1    \\
    0&0 &1 &-1 &-1&-1    \\
    0&0 &-1 &1 &-1&-1    \\
    0&0 &-1 &-1 &1&-1    \\
    0&0 &-1 &-1 &-1&1   \\
    \end{pNiceMatrix}.\] 
\end{example}
    
The determinant of the matrix obtained by deleting the second row of the matrix $(I_{ns} - {A_{Pow^{\ast}(\mathbb{Z}_{2^{n}})}}_{ns})^{tr}$  is equal to
\begin{eqnarray*}
  \det \begin{pmatrix}
  -1 &  -1\\
  -1&1
  \end{pmatrix} \cdot  \overset{n-1}{ \underset{i=2}{\prod}} \det B_{2^i}=-2  \cdot \overset{n-1}{ \underset{i=2}{\prod}}(2-2^i)2^{2^i-1}
\end{eqnarray*} and hence non-zero. Therefore the rank of $(I_{ns} - A_{ns})^{tr}$ is $2^n-2$. Since the $k\times k$ minors of $(I_{ns} - A_{ns})^{tr}$ for this case are in the same form as the ones in the odd case, applying the same argument used in the previous section yields the following result.
\begin{theorem}
\label{thm.pow2n}
     The Grothendieck group of $L_{K}(Pow^{\ast}(\mathbb{Z}_{2^n}))$ is isomorphic to 
\[ \underbrace{\mathbb{Z}_{2}\oplus \cdots \oplus \mathbb{Z}_{2}}_{2^{n+1}-2n-1 \ \mathrm{times}}  \oplus \underbrace{\mathbb{Z}_{4}\oplus \cdots \oplus \mathbb{Z}_{4}}_{n-1 \ \mathrm{times}} \oplus \mathbb{Z}. \]
 \end{theorem}

 \section*{Acknowledgement} We would like to thank the anonymous referee for carefully reading the manuscript and for pointing out several typographical errors and unclear parts. These comments have helped improve the accuracy and readability of the paper.

\bibliographystyle{plainnat}

\end{document}